\documentclass[11pt]{article}
\usepackage{amscd, amsmath, amssymb, amsthm}
\usepackage[all,cmtip]{xy}
\usepackage[pagebackref]{hyperref}

\title{Adjoint functors on the derived category of motives}
\author{Burt Totaro}
\date{  }

\def\Z{\text{\bf Z}}
\def\Q{\text{\bf Q}}

\def\C{\text{\bf C}}
\def\P{\text{\bf P}}
\def\F{\text{\bf F}}

\def\arrow{\rightarrow}

\def\imp{\Rightarrow}

\def\eff{\text{eff}}
\def\tr{\text{tr}}
\def\hocolim{\text{hocolim}}
\def\Hom{\text{Hom}}
\def\et{\text{et}}

\def\Tor{\text{Tor}}
\def\QQ{\overline{\Q}}
\def\Spec{\text{Spec}}

\begin{document}
\maketitle
\newtheorem{theorem}{Theorem}[section]
\newtheorem{corollary}[theorem]{Corollary}
\newtheorem{lemma}[theorem]{Lemma}

\theoremstyle{definition}
\newtheorem{definition}[theorem]{Definition}
\newtheorem{example}[theorem]{Example}

\theoremstyle{remark}
\newtheorem{remark}[theorem]{Remark}

Voevodsky's derived category of motives is the main arena
today for the study of algebraic cycles and motivic cohomology.
In this paper we study
whether the inclusions of three important
subcategories of motives have a left
or right adjoint. These adjoint functors are useful constructions
when they exist, describing the best approximation to an arbitrary motive
by a motive in a given subcategory.
We find a fairly complete picture:
some adjoint functors exist, including
a few which were previously unexplored, while others do not exist
because of the failure of finite generation for Chow groups
in various situations. For some base fields,
we determine exactly which adjoint functors exist.

For a field $k$ and commutative ring $R$, we consider three
subcategories of the derived category of motives, $DM(k;R)$:
the category
$DMT(k;R)$ of {\it mixed Tate motives},
the category
$DM_{\eff}(k;R)$ of {\it effective motives},
and the category $D_0(k;R)$
of (non-effective) motives of dimension $\leq 0$. Each is a
{\it localizing subcategory }of $DM(k;R)$, meaning
a full triangulated subcategory that is closed under arbitrary
direct sums in $DM(k;R)$.
It is a useful formal property of the category $DM(k;R)$
that it contains
the direct sum and the product of an arbitrary set of objects, not
necessarily finite.

In these three cases, Neeman's Brown Representability Theorem
\cite{Neemanduality}
implies that the inclusion $f^*$ of the subcategory has a right
adjoint $f_*$, and that $f_*$ in turn has a right adjoint
$f^{(1)}$:
$$f^*\dashv f_* \dashv f^{(1)}$$
For example, if $f^*$ denotes the inclusion of $DMT(k;R)$ into $DM(k;R)$,
the existence of $f_*$ means that for every motive $M$
in $DM(k;R)$, there is a mixed Tate motive $C(M)$ with a map
$C(M)\arrow M$ that induces an isomorphism on motivic homology.
This functor has been useful,
for example in characterizing mixed Tate motives
as the motives which satisfy the motivic K\"unneth property
\cite[Theorem 7.2]{Totaromotive}. The functor $f^{(1)}$ has probably
not been considered before.

On the other hand, for many fields $k$ and rings $R$, and for
the three subcategories mentioned above,
the sequence of adjoint functors above
cannot be extended to the left or right,
because of various failures of finite generation for motivic cohomology.

For example, for any algebraically closed field $k$ which is not the algebraic
closure of a finite field, we show that the inclusion $f^*$ of $DMT(k;\Q)$
into $DM(k;\Q)$ does not have a left adjoint, using that the Mordell-Weil
group of an elliptic curve over $k$ has infinite rank.
In particular,
it follows that a product of mixed Tate motives need not be mixed
Tate. We deduce that the analogous subcategory
of {\it cellular spectra }in the stable homotopy category
$SH(k)$ is not closed under products
for some fields $k$. (The opposite conclusion
has been announced at least once.)

By results
of Balmer, Dell'Ambrogio, and Sanders, in the case of $DMT(k;R)$
(but not for the other subcategories we consider), $f^*$ has a left
adjoint if and only if it has a three-fold right adjoint
\cite[Theorem 3.3]{BDS}.
So, for many fields $k$ and rings $R$,
the sequence of adjoint functors stops with the three listed above.

By contrast, the Tate-Beilinson conjecture would imply
that the inclusion of $DMT(k;\Q)$ into $DM(k;\Q)$ is a
{\it Frobenius functor }when $k$ is algebraic over a finite field (Theorem
\ref{finitefield}). This is the strong
property that the right
adjoint to the inclusion is also left adjoint to the inclusion
(and so there is an infinite sequence of adjoints).
It is not clear what to
expect when $k$ is a number field, or when $k$ is replaced by
a regular scheme of finite type over $\Z$.

Next, an example by Ayoub, based on Clemens's example of a complex
variety with Griffiths group of infinite rank,
implies that the inclusion of $DM_{\eff}(\C,\Q)$
into $DM(\C,\Q)$ does not have a three-fold right adjoint
\cite[Proposition A.1]{Huber}. The same goes
for any algebraically closed field of characteristic
zero (Theorem \ref{effright}).
We also show that for many fields $k$ and rings $R$,
the inclusion of $DM_{\eff}(k;R)$ into $DM(k;R)$
does not have a left adjoint
(Theorem \ref{effleft}).

An example by Ayoub and Barbieri-Viale, again building
on Clemens's example, implies that the inclusion
of our third subcategory
$D_0(\C;\Q)$ into $DM(\C;\Q)$ does not have a left adjoint
\cite[section 2.5]{AB}. This
can be viewed as showing that certain generalizations
of the Albanese variety do not exist. We give an analogous example
with finite coefficients, showing that
the inclusion of $D_0(k;\F_p)$ in $DM(k;\F_p)$
does not have a left adjoint in many cases (Theorem \ref{d0left}).
These results imply that
the subcategory $D_0(k;R)$ need not be closed under arbitrary products
in $DM(k;R)$, a question that arose during
the construction of the motive of a quotient stack
\cite[after Lemma 8.8]{Totaromotive}. We also show that for many fields $k$
and rings $R$, the inclusion of $D_0(k;R)$ into $DM(k;R)$
does not have a three-fold right adjoint (Theorem \ref{d0right}).

Finally, we prove a positive result: for any scheme $X$ of finite type
over a field $k$ such that the compactly supported
motive $M^c(X)$ in $DM(k;R)$
is mixed Tate, the Chow groups $CH_*(X;R)$ are finitely generated
$R$-modules (Theorem \ref{mixedTatefinite}). This helps to clarify
what it means for a scheme to be mixed Tate.

I thank Bruno Kahn and Tudor P\u{a}durariu for useful conversations.
This work was supported by The Ambrose Monell
Foundation and Friends, via the Institute for Advanced Study,
and by NSF grant DMS-1303105.

\section{Notation}
\label{notation}

Let $k$ be a field.
The {\it exponential characteristic }of $k$ means 1 if $k$ has characteristic
zero, or $p$ if $k$ has characteristic $p>0$. Let $R$ be a commutative
ring in which the exponential characteristic of $k$ is invertible.
Following Cisinski and D\'eglise, the derived category $DM(k;R)$
of motives
over $k$ with coefficients in $R$
is defined to be the homotopy category of $G_m^{\tr}$-spectra
of (unbounded) chain complexes of Nisnevich sheaves with transfers
which are $A^1$-local
\cite[section 2.3]{RO}, \cite[Example 6.25]{CDlocal}.
This is a triangulated category with arbitrary direct sums.
(Voevodsky originally considered the subcategory
$DM^{\text{eff}}_{-}(k)$ of ``bounded above effective motives''
\cite{Voevodskytri}.)
For $k$ perfect, R\"ondigs and \O stv\ae r showed that
the category $DM(k;\Z)$ is equivalent to the homotopy
category of modules over the motivic Eilenberg-MacLane spectrum $H\Z$
in Morel-Voevodsky's stable homotopy category $SH(k)$ \cite[Theorem 1]{RO}.

A separated scheme $X$ of finite type over $k$ determines two motives
in $DM(k;R)$, $M(X)$ (called the motive of $X$) and $M^c(X)$
(called the compactly supported motive of $X$). These two motives
are isomorphic if $X$ is proper over $k$. Also, there are objects
$R(j)$ in $DM(k;R)$ for integers $j$, called {\it Tate motives}.
Here $DM(k;R)$ is a tensor triangulated category with identity
object $R(0)$, and $R(a)\otimes R(b)\cong R(a+b)$ for integers
$a$ and $b$. The motive of projective space
is $M(\P^n_k)\cong \oplus_{j=0}^n R(j)[2j]$.

Voevodsky defined {\it motivic cohomology }and (Borel-Moore)
{\it motivic homology }for
any separated scheme $X$ of finite type over $k$
by 
$$H^j(X,R(i))=\Hom(M(X),R(i)[j])$$
and
$$H_j(X,R(i))=\Hom(R(i)[j],M^c(X))$$
\cite[section 2.2]{Voevodskytri}.
These include the Chow groups of algebraic cycles
with coefficients in $R$,
as $H_{2i}(X,R(i))\cong CH_i(X;R):=CH_i(X)\otimes_{\Z}R$
and $H^{2i}(X,R(i))\cong CH^i(X;R):=CH^i(X)\otimes_{\Z}R$.
More generally, the motivic cohomology and motivic homology
of any object $N$ in $DM(k;R)$ are defined by
$H^j(N,R(i))=\Hom(N,R(i)[j])$ and $H_j(N,R(i))=\Hom(R(i)[j],N)$.

For an equidimensional separated scheme $X$ of dimension $n$
over $k$, motivic homology is isomorphic to Bloch's higher Chow groups:
$$CH^{n-i}(X,j-2i;R)\cong H_j(X,R(i)).$$
It follows that
the motivic homology $H_j(X,R(i))$ of a separated $k$-scheme $X$
is zero unless
$j\geq 2i$ and $j\geq i$ and $i\leq \dim(X)$. The isomorphism
between motivic homology and higher Chow groups was proved
under mild assumptions in \cite[Proposition 4.2.9]{Voevodskytri}; see
\cite[section 5]{Totaromotive} for references to the full statement.

The triangulated category $DM(k;R)$
is compactly generated
\cite[Theorem 11.1.13]{CDtri}, \cite[Proposition 5.5.3]{Kelly}.
(For $k$ imperfect, see \cite[Proposition 8.1]{CDint}.)
A set of compact generators is given by the motives $M(X)(a)$
for smooth projective varieties $X$ over $k$ and integers $a$.
Since $DM(k;R)$ is compactly generated,
it contains arbitrary products
as well as arbitrary direct sums \cite[Proposition 8.4.6]{Neemanbook}.

Define a {\it thick }subcategory of a triangulated category
to be a strictly full triangulated subcategory
that is closed under direct summands. We use the following
result of Neeman's \cite[Theorem 2.1]{Neemanduality}.

\begin{theorem}
\label{compactgens}
Let ${\cal T}$ be a compactly generated triangulated category,
and let ${\cal P}$ be a set of compact generators.
Then any compact object in ${\cal T}$
belongs to the smallest thick
subcategory of ${\cal T}$ that contains ${\cal P}$.
\end{theorem}

\section{Background on triangulated categories}
\label{triang}

We consider three subcategories of $DM(k;R)$ in this paper.
The category $DMT(k;R)$ of {\it mixed Tate motives }is the smallest
localizing subcategory that contains $R(j)$ for all integers
$j$. The category $DM_{\eff}(k;R)$ of {\it effective motives }is
the smallest localizing subcategory that contains $M(X)$
for every smooth projective variety $X$ over $k$. The category
$D_0(k;R)$ of (non-effective) motives of dimension $\leq 0$
is the smallest localizing subcategory that contains $M(X)(-b)$
for every smooth projective variety $X$ over $k$ and every integer
$b\geq\dim(X)$.

We use the following consequences of Neeman's Brown
Representability Theorem \cite[Corollary 2.3]{BDS},
\cite[Theorem 5.1]{Neemanduality}.

\begin{theorem}
\label{adjointcriterion}
Let $F\colon {\cal S}\arrow {\cal T}$ be a exact functor between
triangulated categories, and assume that ${\cal S}$ is compactly
generated. Then:

(1) $F$ has a right adjoint if and only if it preserves arbitrary
direct sums.

(2) $F$ has a left adjoint if and only if preserves arbitrary
products.
\end{theorem}

\begin{theorem}
\label{compact}
Let $F\colon {\cal S}\arrow {\cal T}$ be an exact functor
between triangulated categories with right adjoint $G$,
and assume that ${\cal S}$ is compactly generated.
Then $F$ preserves compact objects if and only if $G$
preserves arbitrary direct sums.
\end{theorem}

The following lemma applies to the three
subcategories of $DM(k;R)$ considered in this paper: mixed Tate motives,
effective motives, and (non-effective) motives of dimension $\leq 0$.

\begin{lemma}
\label{tworight}
Let ${\cal T}$ be a compactly generated triangulated category,
and let ${\cal S}$ be the smallest localizing subcategory
of ${\cal T}$ that contains a given set of compact objects
in ${\cal T}$. Then the inclusion $f^*$ of ${\cal S}$
into ${\cal T}$ has a right adjoint $f_*$. Moreover, $f_*$ also has a right
adjoint $f^{(1)}\colon {\cal S} \arrow {\cal T}$:
$$f^*\dashv f_* \dashv f^{(1)}$$
\end{lemma}

The fact that $f_*$ exists means that
for every object $A$ of ${\cal T}$ there is an object
$B$ of ${\cal S}$ and a morphism $B\arrow A$ that is universal
for maps from objects of ${\cal S}$ to $A$. This is often
a useful construction. In this paper, we ask (in various examples)
whether the inclusion $f^*$ of ${\cal S}$ into ${\cal T}$ also has a left
adjoint $f_{(1)}$. Equivalently, for every object $A$ in ${\cal T}$,
is there an object $B$ of ${\cal S}$ with a map $A\arrow B$
that is universal for maps from $A$ to objects of ${\cal S}$?

The notation $f^{(1)}$
was suggested by Balmer, Dell'Ambrogio, and Sanders
\cite[Corollary 2.14]{BDS}.

\begin{proof}
(Lemma \ref{tworight}) First, because ${\cal S}$ is compactly
generated and the inclusion $f^*$ from ${\cal S}$
to ${\cal T}$ preserves arbitrary direct sums, $f^*$ has a right
adjoint, by Theorem \ref{adjointcriterion}. Next, we use that
the given generators for ${\cal S}$
are compact in ${\cal T}$. It follows that $f^*$ takes compact
objects in ${\cal S}$ to compact objects in ${\cal T}$. Since
${\cal S}$ is compactly generated, it follows that
$f_*$ preserves arbitrary direct sums, by Theorem \ref{compact}.
Since ${\cal T}$ is compactly generated,
Theorem \ref{adjointcriterion} gives
that $f_*$ also has a right adjoint $f^{(1)}$.
\end{proof}

The subcategory $DMT(k;R)$ of $DM(k;R)$ is {\it rigidly-compactly
generated}, unlike $DM_{\eff}(k;R)$ and $D_0(k;R)$. This means
that $DMT(k;R)$ is a tensor-triangulated category;
it has arbitrary direct sums; its compact objects
coincide with the rigid objects (also called the strongly dualizable
objects); and $DMT(k;R)$ is generated by a set of compact
objects. (The key point in checking this
is that the duals in $DM(k;R)$ of the given generators $R(j)$
for $DMT(k;R)$, for integers $j$, are again in $DMT(k;R)$.)

For a tensor exact functor between
rigidly-compactly generated categories that preserves arbitrary
direct sums, Balmer, Dell'Ambrogio,
and Sanders showed that the sequence of adjoint functors 
in Lemma \ref{tworight} extends one step to the left if and only if it
extends one step to the right \cite[Theorem 3.3]{BDS}. In particular:

\begin{theorem}
\label{five}
Let $k$ be a field and $R$ a commutative ring in which the
exponential characteristic of $k$ is invertible. Then
the inclusion $f^*$ of $DMT(k;R)$ into $DM(k;R)$ has a left
adjoint if and only if it has a three-fold right adjoint
(meaning that $f^{(1)}$ above has a right adjoint).
\end{theorem}

\section{The Chow groups of a mixed Tate scheme}

Let $X$ be a scheme of finite type over a field $k$ such that
the compactly supported motive $M^c(X)$ is mixed Tate.
This implies the {\it weak Chow K\"unneth property }that
the Chow groups of $X$ do not increase when the base field $k$
is enlarged \cite[section 6]{Totaromotive}.
However, that leaves open the question
of how big the Chow groups of $X$ are. Note that more general
motivic homology
groups of a mixed Tate scheme $X$ over $k$
need not be finitely generated abelian groups,
as shown by the case $X=\Spec(k)$. (For example,
$H_{-1}(k,\Z(-1))\cong k^*$.)

In this section, we show that for a scheme $X$ of finite type
over a field $k$ such that $M^c(X)$ is mixed Tate in $DM(k;R)$,
the Chow groups $CH_*(X;R)$ are finitely generated $R$-modules.
This was known for the simplest examples of mixed Tate schemes,
{\it linear schemes }over $k$
in the sense of \cite{Totarolinear}. On the other
hand, there are mixed Tate varieties that are not linear schemes
or even rational, for example some Barlow surfaces of general
type \cite[Proposition 1.9]{ACP}, \cite[after Theorem 4.1]{Totaromotive}.

It is natural to ask a stronger question.
Let $X$ be a scheme of finite type that has the weak Chow K\"unneth
property with $R$ coefficients, meaning that $CH_*(X;R)\arrow
CH_*(X_E;R)$ is surjective for all finitely generated fields $E$
over $k$, or equivalently for all fields $E$ over $k$.
Are the Chow groups $CH_*(X;R)$ finitely generated
$R$-modules? 
The answer is yes for $X$ smooth proper over $k$
\cite[Theorem 4.1]{Totaromotive}, but the general question
remains open.

\begin{theorem}
\label{mixedTatefinite}
Let $k$ be a field and $R$ a commutative ring such that the exponential
characteristic of $k$ is invertible in $R$. Let $X$ be a scheme
of finite type over $k$. If $M^c(X)$ is mixed Tate in $DM(k;R)$,
then the Chow groups $CH_*(X;R)$ are finitely generated $R$-modules.
\end{theorem}

\begin{proof}
The object $M^c(X)$ is compact in $DM(k;R)$. Since we assume that
$M^c(X)$ is also mixed Tate (that is, $M^c(X)$ is
in the localizing subcategory
generated by the objects $R(i)$ for integers $i$), it is in fact
in the smallest thick subcategory of $DM(k;R)$ that contains
$R(i)$ for all integers $i$, by Theorem \ref{compactgens}.
In order to see that $X$ has finitely
generated Chow groups, we will analyze which
motives $R(i)[j]$ are needed to construct $M^c(X)$.

Let $N_0=N=M^c(X)$. Consider the following sequence of mixed Tate motives
$N_a$ for $a\geq 0$. Given $N_a$, choose
a set of generators for the motivic homology of $N_a$ as a module
over the motivic homology of $k$. Let $F_a$ be the corresponding direct
sum (possibly infinite) of objects $R(i)[j]$ together with a map
$F_a\arrow N_a$ that induces a surjection on motivic homology.
Let $N_{a+1}$ be a cone of the map $F_a\arrow N_a$. This defines
a sequence of mixed Tate motives $N_0\arrow N_1\arrow \cdots$.

By construction, the homotopy colimit $\hocolim(N_a)$ has zero
motivic homology groups. Since $\hocolim(N_a)$ is a mixed Tate motive,
it follows that $\hocolim(N_a)=0$ (by another of Neeman's results;
see \cite[Corollary 5.3]{Totaromotive}). So
$$0=\Hom(N,\hocolim(N_a))=\varinjlim \Hom(N,N_a).$$
So there is a natural number $a$ such that the composition
$N=N_0\arrow N_1\arrow \cdots \arrow N_a$ is zero. By construction,
the fiber $Y$ of $N=N_0\arrow N_a$ is an iterated
extension of direct sums
of Tate motives, $F_0,\ldots,F_{a-1}$. Since the map $N\arrow N_a$
is zero, $Y$ is isomorphic to $N\oplus N_a[-1]$. Thus $N$ is a summand
of the extension $Y$.

The following lemma formalizes an argument by Neeman
\cite[proof of Lemma 2.3]{Neemansmash}. We say that an object $Y$
in a triangulated category is an {\it iterated extension }of objects
$F_0,\ldots,F_{a-1}$ if there is a map $f_0\colon F_0\arrow Y$,
a map $f_1$ from $F_1$
to the cone of $f_0$, and so on, with the cone of $f_{a-1}$ being zero.

\begin{lemma}
\label{subsum}
Let ${\cal T}$ be a triangulated category with arbitrary direct
sums. Let $N$ be a compact object in ${\cal T}$ which is
a summand of an iterated extension $Y$ of
(possibly infinite) direct sums $F_0,\ldots,F_{a-1}$ of compact objects.
Then $N$ is a summand of an iterated extension
$Y'$ of objects $F_0',\ldots,F_{a-1}'$, with each $F_b'$
a finite direct sum of some of the summands of $F_b$.
\end{lemma}

\begin{proof}
To make an induction, we prove a more general statement.
Let $N$ be a compact object in ${\cal T}$ with a morphism
to an object $Y$, and let $Y'\arrow Y$ be a morphism
whose cone is an iterated extension of direct sums $F_0,\ldots,
F_{a-1}$ of compact objects in ${\cal T}$. Then there is an object $N'$
and a map $N'\arrow N$ such that the composite $N'\arrow N\arrow Y$
factors through $Y'$, and the cone of $N'\arrow N$
is an iterated extension of objects $F_0',\ldots,F_{a-1}'$,
with each $F_b'$
a finite direct sum of some of the given summands of $F_b$.
For $Y'=0$, this gives the statement of the lemma.

The proof is by induction on the number $a$. If $a=1$, then
the cone $F=F_0$ of $Y'\arrow Y$ is a direct sum of compact objects.
Since $N$ is compact, the composition $N\arrow Y\arrow F$
factors through a finite direct sum $F'$ of the given summands of $F$.
Then we can complete the commutative square
$$\xymatrix@C-10pt@R-10pt{
N \ar[r]\ar[d] & F'\ar[d]\\
Y \ar[r] & F
}$$
to a map of triangles
$$\xymatrix@C-10pt@R-10pt{
N'\ar[r]\ar[d] & N \ar[r]\ar[d] & F'\ar[r]\ar[d] & N'[1]\ar[d]\\
Y'\ar[r] & Y \ar[r] & F\ar[r] & Y'[1].
}$$
Thus the cone of $N'\arrow N$ is a finite direct sum $F'$
of the given summands of $F=F_0$, and the composite
$N'\arrow N\arrow Y$ factors through $Y'$, as we want.

Now suppose that $a>1$. Then we can factor the map $Y'\arrow Y$ (with cone
an extension of $F_0,\ldots,F_{a-1}$)
as $Y'\arrow Y''\arrow Y$ such that the cone of $Y'\arrow Y''$ is an
extension of $F_0,\ldots,F_{a-2}$ and the cone of $Y''\arrow Y$ is $F_{a-1}$.
By the case $a=1$ of the induction, there is a map $N''\arrow N$ with
cone a finite subsum $F_{a-1}'$ of $F_{a-1}$
such that $N''\arrow N\arrow Y$
factors through $Y''$. Then $N''$ is compact. By induction on $a$,
there is a map $N'\arrow N''$ with cone
an extension of finite subsums $F_0',\ldots,F_{a-2}'$
of the direct sums $F_0,\ldots,F_{a-2}$
such that $N'\arrow N''\arrow Y''$ factors through $Y'$. Then we have
a commutative diagram
$$\xymatrix@C-10pt@R-10pt{
N'\ar[r]\ar[d] & N'' \ar[r]\ar[d] & N\ar[d] \\
Y'\ar[r] & Y'' \ar[r] & Y,
}$$
which shows that the composite $N'\arrow N\arrow Y$ factors through
$Y'$. Finally, the cone of $N'\arrow N$ is an extension
of $F_0',\ldots,F_{a-1}'$, by the octahedral axiom.
\end{proof}

We showed above that $N=M^c(X)$ is a summand of an extension
$F_0,\ldots,F_{a-1}$ of direct sums of Tate motives.
Since $N$ is compact, 
Lemma \ref{subsum} gives that $N$ is a summand of an
extension $Y'$ of {\it finite }direct sums $F_0',\ldots, 
F_{a-1}'$ of Tate motives,
where each $F_b'$ is the direct sum of finitely many
of the Tate motives that occur in $F_b$.

We now use that for a scheme $X$ of finite type over $k$, the motivic
homology $H_j(X,R(i))$ vanishes unless $2i\leq j$, by section \ref{notation}.
As a result, we can take $F_0$ to be a direct sum of objects
$R(i)[j]$ with $2i\leq j$. Since $N_1$ is a cone of the morphism
$F_0\arrow N_0$ which induces a surjection on motivic homology,
we have an exact sequence of motivic homology groups:
$$\xymatrix@C-10pt@R-10pt{
H_j(N_0,R(i)) \ar[r]_0 & H_j(N_1,R(i)) \ar[r] & H_{j-1}(F_0,R(i)).
}$$
We read off that $N_1$ has a stronger vanishing property than $N_0$ does:
$H_j(N_1,R(i))$ is zero unless $2i-j\leq -1$. Repeating the argument,
we find that each $F_b$ can be chosen to be a direct sum of Tate motives
$R(i)[j]$ with $2i-j\leq -b$.

Therefore, each $F_b'$ is a {\it finite }direct sum
of Tate motives
$R(i)[j]$ with $2i-j\leq -b$. Since $N=M^c(X)$ is a summand of 
the extension $Y'$ of $F_0',\ldots,F_{a-1}'$, we read off that
$CH_*(F_0')\arrow CH_*(X;R)$ is surjective, and that $CH_*(F_0')$
is a finitely generated free $R$-module. Thus the $R$-module
$CH_*(X;R)$ is finitely
generated.
\end{proof}

The same argument gives the following variant. The right adjoint
$f_*$ to the inclusion of $DMT(k;R)$ into $DM(k;R)$ is also
called {\it colocalization }with respect to mixed Tate motives,
$N\mapsto C(N)$.

\begin{theorem}
\label{colocalcompact}
Let $k$ be a field and $R$ a commutative ring such that the exponential
characteristic of $k$ is invertible in $R$. Let $X$ be a scheme
of finite type over $k$. If the colocalization $C(M^c(X))$ 
in $DMT(k;R)$ is compact,
then the Chow groups $CH_*(X;R)$ are finitely generated $R$-modules.
\end{theorem}

\section{Products of mixed Tate motives}
\label{products}

\begin{theorem}
\label{chow}
Let $k$ be a field and $R$ a commutative ring. If the product
$\prod_{m=1}^{\infty} R(0)$ in $DM(k;R)$ is mixed Tate,
then the abelian group $CH_i(Y;R)$ is finitely generated
for every smooth projective variety $Y$ over $k$ and every integer $i$.
\end{theorem}

\begin{proof}
Suppose that $P:=\prod_{m=1}^{\infty} R(0)$ in $DM(k;R)$
is mixed Tate. That implies that for every smooth projective
variety $Y$ over $k$, Dugger-Isaksen's K\"unneth spectral sequence
$$E_2^{pq}=\Tor^{H_*(k,R(*))}_{-p,-q,j}(H_*(P,R(*)),H_*(Y,R(*)))\imp
H_{-p-q}(P\otimes M(Y),R(j))$$
converges to the motivic
homology of $P\otimes M(Y)$
\cite[Proposition 7.10]{DI}. Here,
for bigraded modules $M$ and $N$ over a bigraded ring $S$,
$\Tor^S_{a,i,j}(M,N)$ denotes the $(i,j)$th bigraded piece
of $\Tor^S_a(M,N)$. For this purpose, the group
$H_i^M(X,R(j))$ has bigrading $(i,j)$.

Next, $P\otimes M(Y)$
is isomorphic to $\prod_{m=1}^{\infty} M(Y)$.
(To prove that, use that $M(Y)$ is strongly dualizable
in $DM(k;R)$ (a reference is \cite[Lemma 5.5]{Totaromotive}),
and check that the abelian group of maps from any object $W$
in $DM(k;R)$ to $P\otimes M(Y)$
can be identified with the group of maps
from $W$ to $\prod_{m=1}^{\infty} M(Y)$.)

The motivic homology of $P$ is (trivially) the product
of infinitely many copies of the motivic homology
of $R(0)$. (In particular, $H_i(P,R(j)) = 0$ unless $i\geq 2j$
and $i\geq j$ and $j\leq 0$,
just as we would have for a 0-dimensional variety.)
As a result, the K\"unneth
spectral sequence with $R(j)$ coefficients
is concentrated in columns $\leq 0$ and rows $\leq -2j$.
If we write $H_*(P)$ for the bigraded group $H_*(P,R(*))$, the $E_2$
term looks like:
$$\xymatrix@C-10pt@R-10pt{
0 & 0 & 0 & 0\\
[\Tor_2^{H_*k}(H_*P,H_*Y)]_{2j,j} \ar[rrd] & [\Tor_1^{H_*k}
(H_*P,H_*Y)]_{2j,j} & [H_*P\otimes_{H_*k}H_*Y]_{2j,j} & 0\\
[\Tor_2^{H_*k}(H_*P,H_*Y)]_{2j+1,j} & [\Tor_1^{H_*k}
(H_*P,H_*Y)]_{2j+1,j} & [H_*P\otimes_{H_*k}H_*Y]_{2j+1,j} & 0
}$$
So there are no differentials into or out of the upper right
group, $E_2^{0,-2j}$. We deduce
that the homomorphism
$$CH_*(P) \otimes_R CH_*(Y;R) \arrow CH_*(P\otimes M(Y))
= CH_*\bigg( \prod_{m=1}^{\infty} M(Y)\bigg) $$
is an isomorphism. In particular, it is surjective.

That is,
$$\bigg( \prod_{m=1}^{\infty}R\bigg)
\otimes_R CH_*(Y;R) \arrow \prod_{m=1}^{\infty}
CH_*(Y;R)$$
is surjective. But (by definition of the tensor product of $R$-modules)
any element of the first tensor product maps to a sequence
$(a_1, a_2, \ldots)$ in $\prod_m CH_*(Y;R)$ such that $a_1, a_2, \ldots$
all lie in some finitely generated $R$-submodule of $CH_*(Y;R)$.
So we get a contradiction if $CH_*(Y;R)$ is not finitely generated
as an $R$-module.
\end{proof}

Another proof that $DMT(k;R)$ is not closed under products
in $DM(k;R)$, when there is a $k$-variety whose Chow groups
are not finitely generated, can be given as follows.
By Balmer, Dell'Ambrogio, and Sanders,
the inclusion $f^*$
of $DMT(k;R)$ into $DM(k;R)$ has a left adjoint if and only
if it has a three-fold
right adjoint (Theorem \ref{five} above). This in turn is equivalent
to $f^{(1)}$ preserving arbitrary direct sums (Theorem \ref{adjointcriterion}),
or again to $f_*$ (also called $N\mapsto C(N)$)
preserving compact objects (Theorem \ref{compact}).
By Theorem \ref{colocalcompact},
if that holds, then $CH_*(X;R)$ is a finitely generated
$R$-module for every smooth projective $k$-variety $X$.

Theorem \ref{chow} implies that the subcategory of mixed Tate motives
is not closed under products in $DM(k;R)$, in many cases. For example:

\begin{corollary}
\label{Zcoeffs}
Let $k$ be an algebraically closed field. Then the product
$\prod_{m=1}^{\infty} \Z(0)$ in $DM(k;\Z)$ is not mixed Tate.
In particular,
the subcategory of mixed Tate motives is not closed under
products in $DM(k;\Z)$, and the inclusion $DMT(k;\Z)\arrow
DM(k;\Z)$ does not have a left adjoint or a three-fold right adjoint.
\end{corollary}

\begin{proof}
By Theorem \ref{chow}, to show that
$\prod_{m=1}^{\infty} \Z(0)$ in $DM(k;\Z)$ is not mixed Tate,
it suffices
to give an example of a smooth projective variety $Y$ over $k$
such that $CH_0(Y)$ is not finitely generated as an abelian group.
Since $k$ is algebraically closed, we can take $Y$ to be any elliptic curve
over $k$. Then we have an exact sequence
$$ 0 \arrow Y(k) \arrow CH_0(Y) \arrow \Z \arrow 0.$$
The group of points $Y(k)$ is not finitely generated, because it
has torsion of arbitrarily large order.
Since $DMT(k;\Z)$ is not closed under products
in $DM(k;\Z)$, the inclusion does not have a left adjoint. By Balmer,
Dell'Ambrogio, and Sanders, since $DMT(k;\Z)$ is rigidly-compactly
generated, it follows that the inclusion does not have a three-fold
right adjoint (Theorem \ref{five}).
\end{proof}

We can also consider motives with rational coefficients:

\begin{corollary}
\label{ratmotive}
Let $k$ be an algebraically closed field which is not
the algebraic closure of a finite field. Then the product
$\prod_{m=1}^{\infty} \Q(0)$ in $DM(k;\Q)$ is not mixed Tate.
So the subcategory of mixed Tate motives is not closed
under products in $DM(k;\Q)$,
and the inclusion $DMT(k;\Q)\arrow DM(k;\Q)$
does not have a left adjoint or a three-fold right adjoint.
\end{corollary}

\begin{proof}
By Theorem \ref{chow}, to show that
$\prod_{m=1}^{\infty} \Q(0)$ in $DM(k;\Q)$ is not mixed Tate,
it suffices to find
a smooth projective variety $X$ over $k$ such that $CH_0(X;\Q)$
has infinite dimension as a $\Q$-vector space. Since $k$ is not
the algebraic closure of a finite field, this holds for any 
elliptic curve $X$ over $k$, by Frey and Jarden
\cite[Theorem 9.1]{FJ}. The other statements follow
as in the proof of Corollary \ref{Zcoeffs}.
\end{proof}

By contrast, Theorem \ref{finitefield} shows, under the Tate-Beilinson
conjecture, that for $k$ algebraic over a finite field,
the subcategory of mixed Tate motives {\it is }closed
under products in $DM(k;\Q)$,
and the inclusion $DMT(k;\Q)\arrow DM(k;\Q)$ has a left adjoint.

Finally, we can say something with finite coefficients:

\begin{theorem}
\label{modpmotive}
Let $p$ be a prime number. Then the product
$\prod_{m=1}^{\infty} \F_p(0)$ in $DM(\C;\F_p)$ is not mixed Tate.
So the subcategory of mixed Tate motives is not closed
under products in $DM(\C;\F_p)$,
and the inclusion $DMT(\C;\F_p)
\arrow DM(\C;\F_p)$ does not have a left adjoint or a three-fold
right adjoint.

For any algebraically closed field $k$ of characteristic zero
in place of $\C$, these results hold for all prime numbers $p$
congruent to 1 modulo 3.
\end{theorem}

\begin{proof}
By Theorem \ref{chow}, to show that $DMT(k;\F_p)$ is not closed under
products in $DM(k;\F_p)$,
it suffices to exhibit a smooth
projective variety $X$ over $k$ with $CH_i(X;\F_p)=CH_i(X)/p$
infinite for some $i$.
For $k$ algebraically closed,
$CH_0(X;\F_p)=CH_0(X)/p$ is finite for every smooth
projective variety $X$ over $k$, and so the proof has to be slightly
different from the previous cases. We can instead use Schoen's
theorem that, for $k$ algebraically closed of characteristic zero
and $p\equiv 1\pmod{3}$,
the product $X$ of three copies of the Fermat cubic curve
$x^3+y^3+z^3=0$ over $k$
has $CH_1(X)/p$ infinite \cite[Theorem 0.2]{Schoen}.
(Schoen proves this for $k=\QQ$, and then we can use
Suslin's theorem that $CH_i(X;\F_p)\arrow CH_i(X_F,\F_p)$
is an isomorphism for every algebraically closed field
$F/\QQ$ 
\cite[Corollary 2.3.3]{SuslinICM}.)

Strengthening a result by Rosenschon and Srinivas \cite{RS},
I showed that $CH_1(X)/p$ is infinite for $X$ a very general
principally polarized abelian 3-fold over $\C$
and {\it all }prime numbers $p$ \cite{Totaroinfinite}.
This yields the result
we want over $\C$. The statements about adjoint functors follow
as in the proof of Corollary \ref{Zcoeffs}.
\end{proof}

\section{Products of cellular spectra}

Let $k$ be a field.
Following Dugger-Isaksen, the subcategory of {\it cellular spectra }in
the stable homotopy category $SH(k)$ is the smallest localizing
subcategory that contains the spheres $S^{a,b}$ for all integers
$a$ and $b$ \cite{DI}. Here $S^{1,1}$ is the class of the pointed
curve $(A^1-0,1)$ over $k$, and $S^{1,0}$ is the circle as a simplicial set.
We have $S^{a+1,b}=S^{a,b}[1]$, in terms of the 
structure of $SH(k)$ as a triangulated category. The natural functor
from $SH(k)$ to $DM(k;R)$ takes $S^{a,b}$ to $R(b)[a]$.

\begin{corollary}
Let $k$ be an algebraically closed field which is not
the algebraic closure of a finite field. Then $S^{0,0}_{\Q}$
is cellular in $SH(k)$, but the product
$\prod_{m=1}^{\infty} S^{0,0}_{\Q}$ in $SH(k)$ is not cellular.
So the subcategory of cellular spectra is not closed
under products in $SH(k)$,
and the inclusion
of this subcategory into $SH(k)$ does not have a left
adjoint. It also does not have a three-fold right adjoint.
\end{corollary}

\begin{proof}
Following B\"okstedt and Neeman, 
the {\it homotopy colimit }$X_{\infty}=\hocolim (X_0\arrow X_1
\arrow \cdots)$ in a triangulated category with arbitrary direct
sums is defined as a cone of the morphism
$$1-s\colon \oplus_{i\geq 0}X_i\arrow \oplus_{i\geq 0}X_i,$$
where $s$ is the given map from each $X_i$ to $X_{i+1}$
\cite{BN}.
The spectrum $S^{0,0}_{\Q}$ is cellular in $SH(k)$, because it can be
defined as the homotopy colimit of the sequence 
$$\xymatrix@C-10pt@R-10pt{
S^{0,0} \ar[r]_2 & S^{0,0}\ar[r]_3 & \cdots.
}$$

We can think
of $SH(k;\Q)$ as a full subcategory of $SH(k)$, with the 
rationalization of a spectrum $X$ defined as $X\wedge S^{0,0}_{\Q}$,
or equivalently as the homotopy colimit of 
$$\xymatrix@C-10pt@R-10pt{
X \ar[r]_2 & X \ar[r]_3 & \cdots.
}$$
It is clear that rationalization $SH(k)\arrow SH(k;\Q)$ takes cellular
objects in $SH(k)$ to cellular objects in $SH(k;\Q)$ (meaning objects
in the smallest localizing subcategory of $SH(k;\Q)$ that contains all rational
spheres $S^{a,b}_{\Q}$). 

Suppose that $\prod_{m=1}^{\infty} S^{0,0}_{\Q}$ is cellular
in $SH(k)$. Then the rationalization
$(\prod_{m=1}^{\infty} S^{0,0}_{\Q})_{\Q}$ is cellular in $SH(k;\Q)$.
From the definition of the rationalization as a homotopy colimit,
we see that this rationalization
is simply $\prod_{m=1}^{\infty} S^{0,0}_{\Q}$. We conclude
that $\prod_{m=1}^{\infty} S^{0,0}_{\Q}$
is cellular in $SH(k;\Q)$.

Since $k$ is algebraically closed, $-1$ is a sum of squares in $k$.
Under that assumption, Cisinski and D\'eglise deduced
from Morel's work that $SH(k;\Q)$ is equivalent
to the derived category of motives, $DM(k;\Q)$
\cite[Corollary 16.2.14]{CDtri}. So $\prod_{m=1}^{\infty}
\Q(0)$ is a mixed Tate motive in $DM(k;\Q)$, contradicting
Corollary \ref{ratmotive}. So in fact
$\prod_{m=1}^{\infty} S^{0,0}_{\Q}$
in $SH(k)$ is not cellular. As a result,
the subcategory of cellular spectra is not closed
under products
in $SH(k)$.

As a result, the inclusion $f^*$ of cellular spectra into $SH(k)$
does not have a left adjoint. The inclusion does have a right adjoint
$f_*$, which in turn has a right adjoint $f^{(1)}$,
by Theorem \ref{tworight}.
Since the subcategory of cellular
spectra is rigidly-compactly generated and $f^*$ does not have a left
adjoint, it follows from Balmer, Dell'Ambrogio, and Sanders
that $f^{(1)}$ does not have a right adjoint \cite[Theorem 3.3]{BDS}.
\end{proof}

\section{Effective motives}

Here we show that the inclusion from the subcategory
of effective motives $DM_{\eff}(k;R)$ to $DM(k;R)$
does not have a left adjoint or a three-fold right adjoint,
in many cases. For the three-fold right adjoint, this is
a reformulation of an example by Ayoub. The right adjoint $f_*$
of the inclusion $f^*$ has been used by Huber and Kahn under
the name $\nu_{\leq 0}$ (or step 0 of the {\it slice filtration})
\cite{HK}.

\begin{theorem}
\label{effright}
Let $k$ be an algebraically closed field of characteristic zero.
Let $f^*$ be the inclusion of $DM_{\eff}(k,\Q)$ into $DM(k,\Q)$.
Then the right adjoint $f_*$ of $f^*$ does not preserve compact objects;
the right adjoint $f^{(1)}$ of $f_*$ does not preserve arbitrary direct 
sums; and $f^{(1)}$ does not have a right adjoint:
$$f^*\dashv f_* \dashv f^{(1)}$$
\end{theorem}

\begin{proof}
Ayoub showed that $f_*\colon DM(k,\Q)\arrow DM_{\eff}(k,\Q)$
does not preserve compact objects, for $k$ algebraically closed
of characteristic zero with sufficiently large transcendence
degree. He used Clemens's example
of a complex 3-fold $X$ whose Griffith group has infinite rank
\cite[Proposition A.1]{Huber}. The argument works for any algebraically
closed field of characteristic zero by using instead Schoen's
example of a 3-fold over $\QQ$ whose Griffiths group
has infinite rank \cite{SchoenGriffiths}.
It follows that the right adjoint $f^{(1)}$ of $f_*$ does not preserve
arbitrary direct sums, by Theorem \ref{compact}. Therefore,
$f^{(1)}$ does not have a right adjoint.
\end{proof}

A simpler argument shows that the inclusion $f^*$ from $DM_{\eff}(k;R)$
to $DM(k;R)$ does not have a left adjoint in most cases:

\begin{theorem}
\label{effleft}
Let $k$ be a field, and let $R$ be a commutative noetherian ring in which
the exponential characteristic of $k$ is invertible. If the inclusion
from $DM_{\eff}(k;R)$ to $DM(k;R)$ has a left adjoint,
then every motivic cohomology group $H^j(X,R(i))$
is a finitely generated $R$-module
for every smooth projective variety $X$ over $k$. This fails,
for example, if $R=\Q$ and $k$ is not an algebraic extension
of a finite field; or if $R=\Z$ and $k$ is an infinite field;
or if $R=\F_p$ for a prime number congruent to 1 modulo 3
and $k$ is an algebraically closed field of characteristic
zero; or if $R=\F_p$ for any prime number $p$ and $k=\C$.
\end{theorem}

\begin{proof}
Suppose that the inclusion $f^*$ from $DM_{\eff}(k;R)$ to $DM(k;R)$
has a left adjoint $f_{(1)}$. Since $f^*$ preserves arbitrary direct
sums, $f_{(1)}$ must preserve compact objects, by Theorem \ref{compact}.

Let $X$ be a smooth projective variety over $k$. Let $j$ be an integer.
By the isomorphism
between motivic cohomology and higher Chow groups,
$H^j(X,R(0))$ is isomorphic to $CH^0(X,-j;R)$, which is $R$ if $j=0$
and zero otherwise. Let $N$ be a compact object in $DM_{\eff}(k;R)$.
By Theorem \ref{compactgens}, $N$
belongs to the smallest thick subcategory of $DM(k;R)$ that contains $M(X)$
for all smooth projective varieties $X$ over $k$. Since $R$
is noetherian, the exact sequences for Hom in a triangulated
category yield
that $\Hom(N,R(0))=H^0(N,R(0))$ is a finitely generated $R$-module.

For every object $A$ in $DM_{\eff}(k;R)$, the definition of $f_{(1)}$
gives a map $A\arrow f_{(1)}(A)$ which is universal for maps from $A$
into $DM_{\eff}(k;R)$. In particular, $H^0(f_{(1)}(A),R(0))$ maps
isomorphically to $H^0(A,R(0))$. Let $A$ be compact in $DM(k;R)$;
then $f_{(1)}(M)$ is compact in $DM_{\eff}(k;R)$.
So $H^0(M,R(0))$ is a finitely generated
$R$-module. Since $H^j(X,R(i))\cong H^0(M(X)(-i)[-j],R(0))$ for any
smooth projective variety $X$ over $k$ and integers $i$ and $j$,
it follows that all motivic cohomology groups
of smooth projective varieties with $R$ coefficients are finitely
generated.

It remains to show that this conclusion fails for the pairs $(k,R)$
mentioned in the theorem. First, if $R=\Q$, then 
the $\Q$-vector space $H^1(k,\Q(1))=
k^*\otimes\Q$ has infinite dimension if the field $k$
is not an algebraic extension of a finite field. Next, if $R=\Z$,
then the abelian group $H^1(k,\Z(1))=k^*$ is not finitely generated
if $k$ is an infinite field. Finally, if $R=\F_p$ for a prime
number $p\equiv 1 \pmod{3}$ and $k$ is algebraically closed
of characteristic zero, then Schoen found
a smooth projective 3-fold $X$ over $k$ with $CH^2(X)/p$ infinite
\cite[Theorem 0.2]{Schoen}.
If $R=\F_p$ for any prime number $p$, I exhibited
a smooth complex projective 3-fold $X$
with $CH^2(X)/p$ infinite \cite{Totaroinfinite}.
\end{proof}

\section{The dimension filtration on motives}

Let $D_0(k;R)$ (also called $d_{\leq 0}DM(k;R)$ by analogy with
Voevodsky's notation \cite[section 3.4]{Voevodskytri})
be the smallest
localizing subcategory of $DM(k;R)$ that contains $M(X)(-b)$
for all smooth projective varieties $X$ over $k$ and all integers
$b$ such that $b\geq \dim(X)$. The subcategory
$D_0(k;R)$ was useful for constructing and studying
the compactly supported motive of a quotient stack over $k$,
for example of a classifying stack $BG$ \cite[section 8]{Totaromotive}.

In this section we show that the inclusion of $D_0(k;R)$
into $DM(k;R)$ does
not have a left adjoint or a three-fold right adjoint, in many cases.
Ayoub and Barbieri-Viale gave the first example where the left adjoint
does not exist \cite[section 2.5]{AB}. These examples
imply that the subcategory $D_0(k;R)$
need not be closed under products in $DM(k;R)$, which answers
a question in \cite{Totaromotive}, after Lemma 8.8.

One can think of the nonexistence of a left adjoint
as meaning that certain generalizations
of the Albanese variety do not exist. Indeed, Ayoub and Barbieri-Viale,
generalizing an earlier result by Barbieri-Viale and Kahn,
showed that for a field $k$, the inclusion
$$d_{\leq 1}DM_{\eff}(k;\Q)\arrow DM_{\eff}(k;\Q)$$
has a left adjoint LAlb, related to the Albanese variety of a smooth
projective variety
\cite[Theorem 2.4.1]{AB}.

\begin{theorem}
\label{d0left}
(1) The subcategory $D_0(\C;\Q)$ is not closed under products
in $DM(\C,\Q)$, and
the inclusion functor from $D_0(\C;\Q)$ to $DM(\C;\Q)$
does not have a left adjoint.

(2) Let $k$ be an algebraically closed field of characteristic zero,
and let $p$ be a prime number congruent to 1 modulo 3. Then
the subcategory $D_0(k;\F_p)$ is not closed under
products in $DM(k;\F_p)$, and
the inclusion functor from $D_0(k;\F_p)$ to $DM(k;\F_p)$
does not have a left adjoint. If $k=\C$, then this holds
for any prime number $p$.
\end{theorem}

It would be interesting to find out whether the inclusion
of $D_0(k;R)$ into $DM(k;R)$ has a left adjoint for other fields $k$ and
commutative rings $R$.

\begin{proof}
(1) Ayoub and Barbieri-Viale showed that the inclusion
$$d_{\leq 2}DM_{\eff}(\C;\Q)\arrow DM_{\eff}(\C;\Q)$$
does not have a left adjoint, using Clemens's example
of a complex variety with Griffiths group of infinite
rank \cite[section 2.5]{AB}. (In contrast to Theorem \ref{effright},
it is not clear how to generalize Ayoub and Barbieri-Viale's
argument to arbitrary algebraically closed fields
of characteristic zero.)
The same argument gives
that the inclusion of $D_0(\C;\Q)$ into $DM(\C;\Q)$
does not have a left adjoint.
Equivalently, by Theorem \ref{adjointcriterion},
the subcategory $D_0(\C;\Q)$ is not closed
under products in $DM(\C;\Q)$.

(2) Let $R=\F_p$.
Let $f^*\colon D_0\arrow DM(k;R)$ be the inclusion. Since $D_0(k;R)$
is the smallest localizing subcategory containing a certain set
of compact objects, the inclusion $f^*$ has a right adjoint
$f_*$. Suppose that $f^*$ also has a left adjoint $f_{(1)}$.
Since $f^*$ preserves arbitrary direct sums, $f_{(1)}$
must take compact objects in $DM(k;R)$ to compact
objects in $D_0$, by Theorem \ref{compact}.

Let $X$ be a smooth projective 3-fold over $k$. Then $M(X)(-2)$
is compact in $DM(k;R)$, and so $f_{(1)}(M(X)(-2))$ is compact in $D_0$.
By section \ref{notation},
\begin{align*}
CH^2(X;R) &= H^4(X,R(2))\\
&= \Hom_{DM}(M(X), R(2)[4])\\
&\cong \Hom_{DM}(M(X)(-2)[-4],f^*(R))
\end{align*}
(which makes sense because the object $R$ is in $D_0$)
\begin{align*}
& \cong \Hom_{D_0}(f_{(1)}(M(X)(-2)[-4]), R).
\end{align*}

I claim that $\Hom_{D_0}(N,R)$ is finite for every compact object
$N$ in $D_0$. We know that $N$ can be obtained from the objects
$M(Y)(j)[b]$ with $Y$ smooth projective over $k$, $b\in\Z$
and $j+\dim(Y)\leq 0$ by finitely many cones and taking a summand.
So it suffices to show that $\Hom_{D_0}(M(Y)(j)[b],R)$ is finite
for every smooth projective variety $Y$ 
over $k$, $b\in\Z$, and $j+\dim(Y)\leq 0$.
Equivalently, we want to show that the motivic cohomology
group $H^b(Y,R(a))$ is finite
for all smooth projective varieties $Y$ over $k$, all $b\in\Z$, and all
$a\geq \dim(Y)$. This was proved by Suslin: the group mentioned
is isomorphic to etale cohomology $H^b_{\et}(Y,\Z/p(a))$ and hence
is finite, using that $k$ is algebraically closed
\cite[Corollary 4.3]{Suslin}.

Thus, by two paragraphs back, $CH^2(X)/p$ is finite for every
smooth projective 3-fold $X$ over $k$. This contradicts
the fact that there is a smooth projective
3-fold $X$ over $k$ with $CH^2(X)/p$ infinite, under
our assumptions on $k$ and $p$ \cite[Theorem 0.2]{Schoen},
\cite{Totaroinfinite}. We conclude
that the inclusion of $D_0(k;R)$ into $DM(k;R)$
does not have a left adjoint.
\end{proof}

A simpler argument shows that the inclusion $f^*$ from $D_0(k;R)$
to $DM(k;R)$ does not have a three-fold right adjoint in most cases:

\begin{theorem}
\label{d0right}
Let $k$ be a field, and let $R$ be a commutative noetherian ring in which
the exponential characteristic of $k$ is invertible.
Suppose that there is a smooth projective $k$-variety such that
some motivic cohomology group $H^j(X,R(i))$
is not finitely generated as an $R$--module.
Let $f^*$ be the inclusion of $D_0(k,R)$ into $DM(k,R)$.
Then the right adjoint $f_*$ of $f^*$ does not preserve compact objects;
the right adjoint $f^{(1)}$ of $f_*$ does not preserve arbitrary direct 
sums; and $f^{(1)}$ does not have a right adjoint:
$$f^*\dashv f_* \dashv f^{(1)}$$

These negative results hold, for example,
if $R=\Q$ and $k$ is not an algebraic extension
of a finite field; or $R=\Z$ and $k$ is an infinite field;
or $R=\F_p$ for any prime number $p$ and $k=\C$;
or $R=\F_p$ with $p$ a prime number congruent to 1 modulo 3
and $k$ is an algebraically closed field of characteristic zero.
\end{theorem}

\begin{proof}
Suppose that there is a smooth projective variety $X$ over $k$
such that some motivic cohomology group $H^j(X,R(i))$
is not finitely generated as an $R$-module.
We will show that the right adjoint
$f_*\colon DM(k;R)\arrow D_0(k;R)$
does not preserve compact objects. Given that,
the right adjoint $f^{(1)}$ of $f_*$ does not preserve
arbitrary direct sums, by Theorem \ref{compact}. Therefore,
$f^{(1)}$ does not have a right adjoint.

If $f_*$ preserves compact objects, then for every compact object
$M$ in $DM(k;R)$, we have a compact object $f_*M$ in $D_0(k;R)$
and a map $f_*M\arrow M$ which is universal for maps from $D_0(k;R)$
to $M$. In particular, since $R(0)$ is in $D_0(k;R)$,
$H_0(f_*M,R(0))\arrow H_0(M,R(0))$ is a bijection.

Let $X$ be a smooth projective variety of dimension $n$,
and let $b$ be an integer
such that $b\geq n$. (The objects $N=M(X)(-b)$ of this form
generate $D_0(k;R)$.) I claim that the $R$-module
$H_0(N[-j];R(0))$ is finitely generated for all integers $j$.
This group is $H_j(X,R(b))$. By the isomorphism of motivic
homology with higher Chow groups (see section \ref{notation}),
this group is zero if $b>n$, and
\begin{align*}
H_j(X,R(n)) &\cong CH^0(X,j-2n;R)\\
&\cong \begin{cases} R & \text{if }j=2n\\
0 & \text{otherwise.}
\end{cases}
\end{align*}
Thus $H_0(N[-j];R(0))$ is either $0$ or $R$, and hence is a finitely
generated $R$-module.

Every compact object in $D_0(k;R)$ belongs to the smallest
thick subcategory that contains $M(X)(-b)$ for all smooth projective
varieties $X$ over $k$ and all $b\geq \dim(X)$ (Theorem \ref{compactgens}).
Therefore,
the long exact sequences for Hom in a triangulated category,
plus the fact that $R$ is noetherian, yield that the $R$-module
$H_0(N,R(0))$ is finitely generated for all
compact objects $N$ in $D_0(k;R)$. If $f^{(1)}$ has a right
adjoint, then (as explained above) it would follow that
the $R$-module $H_0(N,R(0))$ is finitely generated
for all compact objects $N$ in $DM(k;R)$. In particular,
all motivic homology groups of smooth projective $k$-varieties
with $R$ coefficients would be finitely generated, as we want.

Finite generation of motivic cohomology
fails for the pairs $(k,R)$ mentioned in the theorem,
by the proof of Theorem \ref{effleft}.
\end{proof}

\section{Mixed Tate motives over finite fields}

We now show that some of the questions in this paper would have a different
answer for $k$ algebraic over a finite field,
assuming the Tate-Beilinson conjecture.
I do not know what to expect over number fields $k$, or with $k$
replaced by a regular scheme of finite type over $\Z$.

Let $p$ be a prime number. The {\it strong Tate conjecture }over $\F_p$
says that for smooth projective varieties $X$ over $\F_p$ and a prime
number $l\neq p$, the generalized eigenspace for the eigenvalue 1
of Frobenius on $H^{2i}(X_{\overline{\F_p}},\Q_l(i))$ is spanned by
codimension-$i$ algebraic cycles on $X$ with $\Q_l$ coefficients.
The {\it Tate-Beilinson conjecture }over $\F_p$
is the combination of the strong Tate conjecture over $\F_p$
with the conjecture that rational and numerical equivalence coincide,
for algebraic cycles with $\Q$ coefficients on smooth projective
varieties over $\F_p$.

\begin{theorem}
\label{finitefield}
Let $k$ be an algebraic extension field of $\F_p$.
Assume the Tate-Beilinson conjecture. Then the inclusion $f^*$
of the subcategory $DMT(k;\Q)$ into $DM(k;\Q)$ is a Frobenius functor.
That is, the right adjoint functor $f_*$ from $DM(k;\Q)$ to $DMT(k;\Q)$
is also left adjoint to $f^*$. It follows that the subcategory
$DMT(k;\Q)$ is closed under both direct sums and products
in $DM(k;\Q)$.
\end{theorem}

Thus, given Tate-Beilinson,
there is an infinite sequence of adjoint functors,
consisting of $f^*$ and $f_*$ in turn:
$$\cdots \dashv f^* \dashv f_* \dashv f^* \dashv f_* \dashv \cdots.$$

As far as I know,
the Bass conjecture (that $K$-groups of smooth varieties
over $\F_p$ are finitely generated) would not be enough to imply
that $f^*$ has a left adjoint. In particular, Bruno Kahn explained
to me that the Bass conjecture is not known to imply Parshin's conjecture,
which is needed for the following argument. By contrast, the analog
of the Bass conjecture for etale motivic cohomology would
imply Parshin's conjecture.

\begin{proof}
Let $k$ be an algebraic extension field of $\F_p$,
and let $X$ be a smooth projective variety over $k$.
Given the Tate-Beilinson conjecture, the Chow groups
$CH^*(X,\Q)$ are finite-dimensional
$\Q$-vector spaces (and in fact $\dim_{\Q}CH^i(X,\Q)\leq 
\dim_{\Q_l}H^{2i}_{\et}(X_{\overline{k}},\Q_l)$). Also,
Geisser showed that the Tate-Beilinson conjecture
implies Parshin's conjecture that $K_i(X)\otimes\Q=0$ for $i>0$
\cite[Theorem 1.2]{Geisser}. Equivalently,
$H^j(X,\Q(i))=0$ for $j\neq 2i$.

Let $f^*\colon DMT(k;\Q)\arrow DM(k;\Q)$ be the inclusion. Since $DMT(k;\Q)$
is the smallest localizing subcategory containing a certain set
of compact objects, the inclusion $f^*$ has a right adjoint
$f_*$ (by Lemma \ref{tworight}). We also write $N\mapsto C(N)$
for $f_*$.
To prove that $f^*$ also has a left adjoint $f_{(1)}$, it suffices
to show that $f^*$ has a three-fold right adjoint by Theorem
\ref{five}. Equivalently, we have to show that $f^{(1)}$ preserves
arbitrary direct sums (Theorem \ref{adjointcriterion}), or again
that $f_*$ (also called $N\mapsto C(N)$)
preserves compact objects (Theorem \ref{compact}).

The subcategory of compact objects in $DM(k;\Q)$ is the smallest
thick subcategory that contains $M(X)(b)$ for all smooth projective
varieties $X$ over $k$ and all integers $b$. So it suffices
to show that $C(M(X)(b))$ is compact under these assumptions.
Since $DMT(k;\Q)$ is closed under tensoring with $\Q(b)$, it suffices
to show that $C(M(X))$ is compact for every smooth projective $k$-variety
$X$.

As discussed above,
our assumptions give that the $\Q$-vector space
$H^j(X,\Q(i))$ is zero if $j\neq 2i$ and
finite-dimensional if $j=2i$. Also, Quillen's calculation
of the $K$-theory of finite fields \cite[Theorem 8]{Quillen}
gives that $\Hom_{DM(k;\Q)}(\Q(0),\Q(i)[j])$ is $\Q$ if $i=j=0$
and zero otherwise. It follows that there is
a finite direct sum $N$ of Tate motives $\Q(i)[2i]$ and a morphism
$N\arrow M(X)$ that induces an isomorphism on motivic homology
groups. So $N$ is isomorphic to $C(M(X))$, and we have shown
that $C(M(X))$ is compact. This completes the proof
that the inclusion of $DMT(k;\Q)$ into $DM(k;\Q)$ has a left
adjoint as well as a right adjoint.

Finally, we want to show that $f^*$
is a Frobenius functor, that is, that
the right adjoint $f_*$ to the inclusion $f^*$ is also left adjoint
to $f^*$. We know from Lemma \ref{tworight} that $f_*$ has a right
adjoint $f^{(1)}$. Recall that we use the notation $N\mapsto C(N)$
for $f_*$. By Balmer, Dell'Ambrogio, and Sanders, the object
$\omega_f=f^{(1)}(\Q(0))$ (the {\it relative dualizing object }for $f^*$)
is characterized by the existence of a natural bijection
$$\Hom_{DMT}(C(N),\Q(0))\cong \Hom_{DM}(N,\omega_f)$$
for all $N$ in $DM(k;\Q)$ \cite[Definition 1.4]{BDS}. Given
that $f^*$ has a left adjoint $f_{(1)}$, $f^*$ is
a Frobenius functor if and only if $\omega_f\cong \Q(0)$
\cite[Remark 1.15]{BDS}.

Thus, it suffices to show that for $N$ in $DM(k;\Q)$, the map
$C(N)\arrow N$ induces a bijection $H^0(N,\Q(0))\arrow H^0(C(N),\Q(0))$.
Let ${\cal S}$ be the full subcategory of objects $N$ such that
$H^0(C(N)[j],\Q(0))\arrow H^0(N[j],\Q(0))$ is a bijection for all integers $j$.
Clearly ${\cal S}$ is a triangulated subcategory. Also,
$N\mapsto C(N)$ preserves
arbitrary direct sums, by Theorems \ref{adjointcriterion} and \ref{tworight}.
It follows that ${\cal S}$
is a localizing subcategory, using that 
$H^0(\oplus N_{\alpha},\Q(0))\cong \prod H^0(N_{\alpha},\Q(0))$
for any set of objects $N_{\alpha}$. So ${\cal S}$ is equal to $DM(k;\Q)$
as we want if ${\cal S}$ contains $M(X)(-b)$ for all smooth projective
varieties $X$ and all integers $-b$.

To prove this, we use that, by the analysis of $C(M(X))$ above,
the motive $N=M(X)(-b)[-c]$ for integers $b$ and $c$ satisfies
$$C(N)\cong \oplus_j \Q(j-b)[2j-c]\otimes CH_j(X,\Q).$$
We have 
\begin{align*}
H^0(N,\Q(0)) &\cong H^c(X,\Q(b))\\
 &\cong \begin{cases} 0 & \text{if }c\neq 2b\\
CH^b(X;\Q) & \text{if }c=2b.
\end{cases}
\end{align*}
On the other hand, by the description of $C(N)$ above,
$$H^0(C(N),\Q(0))\cong \begin{cases} 0 & \text{if }c\neq 2b\\
CH_b(X;\Q)^* & \text{if }c=2b.
\end{cases}$$
Since rational and numerical equivalence coincide (by the
Tate-Beilinson conjecture),
the natural map $CH^b(X;\Q)\arrow CH_b(X;\Q)^*$ is a bijection.
This shows that $M(X)(-b)$ is in the subcategory ${\cal S}$
for all smooth projective varieties $X$ over $k$ and all integers $b$.
As a result, ${\cal S}$ is equal to $DM(k;\Q)$. That is, the inclusion
from $DMT(k;\Q)$ into $DM(k;\Q)$ is a Frobenius functor.
\end{proof}


\small \sc UCLA Mathematics Department, Box 951555,
Los Angeles, CA 90095-1555

totaro@math.ucla.edu
\end{document}